\documentclass[11pt,reqno]{amsart}
\usepackage[margin=1in,letterpaper]{geometry}                % See geometry.pdf to learn the layout options. There are lots.
\usepackage{graphicx}
\usepackage{amsmath,amssymb}
\usepackage{epstopdf}
\usepackage[colorlinks=true, pdfstartview=FitV, linkcolor=blue, citecolor=blue, urlcolor=blue]{hyperref}
\usepackage{xcolor}
\usepackage[nameinlink]{cleveref}
    \crefname{conj}{conjecture}{conjectures}
    \crefname{conj}{Conjecture}{Conjectures}

\numberwithin{equation}{section}

\newtheorem{thm}{Theorem}[section]

\newtheorem*{introthm*}{Theorem}

\newtheorem{cor}[thm]{Corollary}
\newtheorem{lem}[thm]{Lemma}
\newtheorem{prop}[thm]{Proposition}

\theoremstyle{definition}
\newtheorem{defn}[thm]{Definition}
\newtheorem*{defn*}{Definition}

\newtheorem{ex}[thm]{Example}

\theoremstyle{remark}
\newtheorem{rem}[thm]{Remark}
\newtheorem{notation}[thm]{Notation}

\newcommand{\A}{{\mathbb A}}

\newcommand{\N}{{\mathbb N}}

\newcommand{\Q}{{\mathbb Q}}

\newcommand{\cF}{{\mathcal F}}

\newcommand{\fq}{{\mathfrak q}}

\def\ch{\operatorname{char}}
\def\codim{\operatorname{ht}}
\def\gin{\operatorname{gin}}
\def\ginr{\operatorname{gin_{rev}}}
\def\init{\operatorname{in}}

\def\sreg{\operatorname{sreg}}
\def\red{\operatorname{red}}
\def\reg{\operatorname{reg}}
\def\GL{\operatorname{GL}}

\def\Ass{\operatorname{Ass}}

\newcommand\MyUnderLineWithNoLinebreaks[5]{%
  \protect\leavevmode
  {%
    \color{#2}%
    \vtop{%
      \hbox{\color{#3}#1}%
      \kern-\prevdepth
      \kern#4\relax
      \hrule height #5\relax
      \kern-#4\relax
      \kern-#5\relax
    }%
  }%
}%

\title[Axial constants and sectional regularity of homogeneous ideals]{Axial constants and sectional regularity \\of homogeneous ideals}

\author[\tiny{DeBellevue, Lebovitz, Li, Lotfi, Mohite, Pan, Pathak, Roshan Zamir, Seceleanu, Zhang}]{M.~DeBellevue, A.~Lebovitz, Y.~Li,  M.~Lotfi, S.~Mohite, \\ X.~Pan, M.~S.~Pathak,  S.~Roshan Zamir, A.~Seceleanu, X.~Zhang}

\address[Michael DeBellevue]{University of Nebraska-Lincoln}
\email{michael.debelelvue@huskers.unl.edu}
\address[Audric Lebovitz]{University of Michigan}
\email{alebovit@umich.edu}
\address[Yik Li]{Purdue University}
\email{li4536@purdue.edu}
\address[Mohamed Lotfi]{Wagner College}
\email{mohamed.lotfi@wagner.edu}
\address[Shivam Mohite]{Vanderbilt University}
\email{shivam.j.mohite@vanderbilt.edu}
\address[Xin Pan]{University of Rochester}
\email{xpan9@u.rochester.edu}
\address[Mrigank Shekhar Pathak]{Institute of Mathematics and Applications, India}
\email{masterofscience642@gmail.com}
\address[Shah Roshan Zamir]{University of Nebraska-Lincoln}
\email{sroshanzamir2@huskers.unl.edu}
\address[Alexandra Seceleanu]{University of Nebraska-Lincoln}
\email{aseceleanu@unl.edu}
\address[Sindy Xin Zhang]{University at Buffalo}
\email{xzhang99@buffalo.edu}

\date{}

\begin{document}

\begin{abstract}
We introduce a notion of sectional regularity for a homogeneous ideal $I$, which measures the regularity of its general sections with respect to linear spaces of various dimensions. It is related to axial constants defined as the intercepts on the coordinate axes of the set of  exponents of monomials in the reverse lexicographic generic initial ideal of $I$. We show the equivalence of these notions and several other homological and ideal-theoretic invariants. We also establish that these equivalent invariants grow linearly for the family of powers of a given ideal. 
\end{abstract}

\maketitle

%\setcounter{tocdepth}{1}
%\tableofcontents

%\blfootnote{This work was done as part of the Polymath Jr program partially supported by NSF award DMS--2113535. A.S. is partially supported by NSF award DMS--2101225.}

%\vspace{-2em}

\section{Introduction}

The generic initial ideal (gin) of a homogeneous ideal $I$ of a polynomial ring captures a number of important features of $I$. When computed with respect to the reverse lexicographic monomial order, $\ginr(I)$ retains two of the most significant measures for the homological complexity of $I$: the projective dimension  and Castelnuovo-Mumford  regularity, as shown by Bayer and Stillman in \cite{BS}. These groundbreaking results have subsequently been refined and generalized by Bayer--Charalambous--Popescu \cite{BCP} and Aramova--Herzog \cite{AH} to show that $\ginr(I)$ retains also the extremal Betti numbers of $I$; see \Cref{ex:introex} for an illustration of this notion.

In this paper we enlarge the scope of the algebraic invariants which  can be read off generic initial ideals by focusing on invariants termed {\em axial constants}, a terminology taken from \cite{Walker}. These constants are defined below and  have been studied implicitly by Mayes \cite{Mayes, Mayes2}, Dumnicki--Szemberg--Szpond--Tutaj-Gasi\'nska \cite{DST, DST2}, and Chauhan--Walker \cite{Walker}, from an asymptotic perspective. 

\begin{defn}
\label{def:axial}
For a homogeneous ideal $I$ of the polynomial ring $R=k[x_1, \ldots, x_d]$ and an integer $1\leq i\leq d$  define  the $i$-th {\em axial constant} of $I$ to be 
\[
a_i(I)=\min\{j \in \N \mid x_i^j \in \ginr(I)\}
\]  
If no power of $x_i$ is contained in $\ginr(I)$ then we set $a_i(I)=\infty$.
\end{defn}

Our main contribution is to point out a multitude of ways in which axial  constants can be interpreted as measures of the homological complexity of $I$.  For a glimpse of their significance, note that axial constants simultaneously generalize the notion of initial degree and that of Castelnuovo-Mumford regularity. Indeed, it has  been noted in \cite{Mayes, DST, Walker} that $a_1(I)$  is the least degree of a nonzero element of $I$.  Moreover, in characteristic zero, if the codimension of $I$ is $\codim(I)=c$ and $R/I$ is Cohen-Macaulay then one has $a_c(I)=\reg(R/I)+1=\reg(I)$.

One aim of this paper is to vastly generalize these observations to give homological and ring-theoretic interpretations for the intermediate axial constants $a_i(I)$ with $1<i<\codim(I)$. In the process we introduce the notion of {\em sectional regularity} for homogeneous ideals.

\begin{defn}%[\Cref{def:sreg}]
\label{def:sreg}
Let $I$ be a homogeneous ideal of a polynomial ring $R=k[x_1,\ldots, x_d]$ and let $i$ be an integer $0\leq i\leq d$. The $i$-th {\em sectional regularity} number of the ideal $I$ is defined for sufficiently general linear forms $\ell_{i+1},\ldots,\ell_d$ as
\[
\sreg_i(I)=\reg(I + (\ell_{i+1},\ldots,\ell_d)).
\]
\end{defn}

One of our main results, \Cref{thm:equiv}, shows that there is an equality $\sreg_i(I)=a_i(I)$ for $1\leq i\leq \codim(I)$. However sectional regularity distinguishes itself as a better behaved invariant for the range $\codim(I)<i\leq d$, where it remains finite, in contrast to the axial numbers. A multitude of additional invariants are shown to be equivalent to the two discussed above in \Cref{thm:equiv}.

We provide a formula for sectional regularity in terms of generic annihilator numbers $\alpha_{ij}(R/I)$. This notion, introduced in \cite{AH}, is defined as follows: for a module $M$ set $M^{(i)}=M/ (\ell_{1},\ldots,\ell_{i-1})$, where $\ell_1, \ldots, \ell_d$ are linear forms having the property 
\[
\alpha_{ij}(M)=\dim_k(0:_{M^{(i)}} \ell_i)_j<\infty \text{ for } 0\leq i\leq d-1, j\in \N.
\]
Such linear forms exist by the theory of almost regular sequences introduced in \cite{STT} and developed in \cite{AH}, among others. We rely heavily on this useful tool following in the footsteps of similar applications in \cite{T2, HT}. We establish in \Cref{cor:sregmax} the identity
\[
\sreg_i(I)=\max\{j \mid \alpha_{tj}(R/I)\neq 0, \ t \geq d-i\}+1.
\]
This allows us to formulate a theory of {\em coextremal annihilator numbers} which complements the more familiar notion of extremal annihilator numbers from \cite{AH}. In detail, we call $\alpha_{ij}(R/I)$ a coextremal annihilator number if it satisfies $\alpha_{ij}(R/I)\neq 0$ and 
\[
\alpha_{st}(R/I)=0 \text{ for all pairs }  (s,t)\neq (i,j) \text{ with } s \geq i,t \geq j.
\]
An coextremal annihilator number sits at a south-east corner of the region formed by the nonzero entries in the table of annihilator numbers. For comparison, an extremal annihilator number satisfies $\alpha_{ij}(R/I)\neq 0$ and $\alpha_{st}(R/I)=0$  for $ (s,t) \neq (i,j), s \leq i, t \geq j$, thus marking a south-west corner. We illustrate these notions in connection to axial constants in the following example.

\begin{ex}
\label{ex:introex}
Let $I=(\MyUnderLineWithNoLinebreaks{$x_1^2$}{black}{black}{3pt}{2pt}, x_1x_2, x_1x_3, x_1x_4, x_1x_5,x_1x_6, \MyUnderLineWithNoLinebreaks{$x_2^3$}{black}{black}{3pt}{2pt}, x_2^2x_3, x_2x_3^3,x_2x_3^2x_5,\MyUnderLineWithNoLinebreaks{$x_3^5$}{black}{black}{3pt}{2pt})$. This is an ideal of the polynomial ring $R=\Q[x_1, x_2, x_3, x_4, x_5]$ satisfying $I=\ginr(I)$. 

The tables below present the Betti numbers and the generic annihilator numbers of $R/I$. The extremal Betti numbers are boxed in the first table. The corresponding extremal generic annihilator numbers are boxed in the second table. This correspondence, established in \cite{BCP, AH}, is given by  
$\beta_{i,i+j}(R/I)=\alpha_{d-i,j}(R/I)$ for $d=\dim(R)$ and for each extremal annihilator number $\alpha_{i,j}(R/I)$.

In this paper we focus on the dual notion of coextremal annihilator numbers, which are underlined in  the second table. In \Cref{cor:axialvalues} we show that their position corresponds to a distinguished subset of the minimal  generators for the monomial ideal $I$, namely those minimal generators which are pure powers: $x_1^2, x_2^3, x_3^5$ and thus in turn to the values of the axial constants $a_1(I)=2, a_2(I)=3, a_3(I)=5$. The contributions of these generators to the Betti table of $R/I$, specifically to the homological degree one column, is also highlighted by underlining.
 
\[
\begin{matrix}
     \beta_{i,i+j}(R/I) &0&1&2&3&4&5&6\\
     % \text{total:}&1&11&25&26&16&6&1\\
      \text{0:}&1&\text{.}&\text{.}&\text{.}&\text{.}&\text{.}&\text{.}\\
      \text{1:}&\text{.}&\MyUnderLineWithNoLinebreaks{6}{black}{black}{1pt}{2pt}&15&20&15&6&\fbox{1}\\
      \text{2:}&\text{.}& 2&3&1&\text{.}&\text{.}&\text{.}\\
      \text{3:}&\text{.}&\MyUnderLineWithNoLinebreaks{2}{black}{black}{1pt}{2pt}&5&4& \fbox{1}&\text{.}&\text{.}\\
      \text{4:}&\text{.}&\MyUnderLineWithNoLinebreaks{1}{black}{black}{1pt}{2pt}&2& \fbox{1}&\text{.}&\text{.}&\text{.}
      \end{matrix}
     \qquad      \qquad
  \begin{matrix}
     \alpha_{i,j}(R/I) &0&1&2&3&4&5\\
     % \text{total:}&1&11&25&26&16&6&1\\
      \text{0:}&\text{.}&\text{.}&\text{.}&\text{.}&\text{.}&\text{.}\\
      \text{1:}& \fbox{1}&1&1&1&1&\MyUnderLineWithNoLinebreaks{1}{black}{black}{1pt}{2pt} \\
      \text{2:}&\text{.}& \text{.} & \text{.} &1&\MyUnderLineWithNoLinebreaks{1}{black}{black}{1pt}{2pt}& \text{.}\\
     \text{3:}&\text{.}& \text{.} &\fbox{1}& 1& \text{.} & \text{.}\\
       \text{4:}&\text{.}& \text{.} &  \text{.} &\fbox{\MyUnderLineWithNoLinebreaks{1}{black}{black}{1pt}{2pt}}& \text{.} & \text{.}
      \end{matrix}     
      \]
      
 Note that the same annihilator number can simultaneously be extremal and coextremal; this is true of $\alpha_{3,4}(R/I)$ in the above example.
\end{ex}

Beyond the setting of monomial ideals \Cref{cor:axialvalues} extends the correspondence illustrated in \Cref{ex:introex}  to a correspondence between coextremal annihilator numbers and axial constants of arbitrary homogeneous ideals given by
\[
\alpha_{d-i,j}(R/I) \text{ is  coextremal}  \iff a_{i-1}(I)<a_i(I)=j+1.
\]
This correspondence does not pin down the values of the coextremal annihilator numbers. It would be an interesting pursuit to understand what these values measure.

We now preview the organization of our paper and its main results. After providing necessary background on generic initial ideals in section \ref{s:background}, we introduce the technical underpinnings of sectional regularity in section \ref{appendix} and we express this invariant in terms of coextremal annihilator numbers in \Cref{cor:coextremal}.  Section \ref{s:2} is dedicated to a plethora of invariants, including the axial constants, which are equivalent to the sectional regularity. Our main result stating this equivalence is \Cref{thm:equiv}. In section \ref{s:linear} we study the growth of the sectional regularity and axial constants for powers of ideals. As it was shown by Kodiyalam \cite{Kod} and Cutkosky--Herzog--Trung \cite{CHT} that the Castelnuovo-Mumford regularity is given eventually by a linear function $\reg(I^n)=dn+e$ for $n\gg 0$ for the powers of a homogeneous ideal $I$, it is natural to ask whether the same is true for the sectional regularity and axial constants. We show in \Cref{thm:linear} that the similar result does indeed hold. %Furthermore, we identify the leading terms of these linear functions in the case when $I$ is a complete intersection, giving a new proof for the main result of  \cite{Mayes}.

%Our work focuses primarily on ideals and cyclic modules. However, the notions of sectional regularity and axial constants can be defined in an analogous manner for arbitrary $\N$-graded modules. We leave the details of this generalization to the interested reader.

\section{Background}
\label{s:background}

Let $I$ be a homogeneous ideal of a polynomial ring $R=k[x_1, \ldots, x_d]$. Throughout this paper we assume that $k$ is an infinite field, which is a necessary condition for generic initial ideals to be well-defined. The remarkable notion of generic initial ideals was introduced by Galligo in characteristic zero \cite{Galligo} and Bayer and Stillman in arbitrary characteristic \cite{BS}. 

Fix a monomial order on $R$ so that $x_1 > x_2 > \cdots > x_d$. A monomial order that will be used predominantly in this paper is the graded reverse lexicographic order, for which $x_1^{a_1}\cdots x_d^{a_d}>x_1^{b_1}\cdots x_d^{b_d}$ if and only if $\sum_{i=1}^d a_i>\sum_{i=1}^d b_i$ or $\sum_{i=1}^d a_i=\sum_{i=1}^d b_i$  and for the largest index $i$ so that $a_i\neq b_i$ we have $a_i<b_i$. 
The largest monomial with respect to a fixed order among those whose coefficients are nonzero in  a polynomial $f$ is called the {\em leading term} of $f$ and denoted $in(f)$.  The initial ideal of $I$ with respect to this order is 
$
\init(I)=\left( \init(f) \mid f\in I \right).
$

The natural action of the general linear group $\GL_d(k)$ on $R_1\cong k^d$ extends to an action  on $R$ via $g\cdot f(x_1, \ldots, x_d)= f(\sum_{j=1}^d g_{1j} x_j, \ldots, \sum_{j=1}^d g_{dj} x_j)$ if $g=(g_{ij})$. Thus each element of $g\in \GL_d(k)$ yields an automorphism of $R$, which we call a  linear change of coordinates. We denote the image of $I$ under the corresponding automorphims by $g\cdot I$. The generic initial ideal of $I$ is obtained by first applying to $I$ a sufficiently general linear change of coordinates followed by taking the initial ideal. %To make the terminology ``sufficiently general" precise we view $\GL_d(k)$ as a subset of the affine space $\A^{d^2}_k$ and dub a linear change of coordinates sufficiently general if it is determined by an element $g$ of the Zariski open set $U_{gin}(I)$ in \Cref{def:gin}. 
The details of the construction of the generic initial ideal are as follows.

\begin{defn}[{\cite{BS, Galligo}}]
\label{def:gin}
Fix a homogeneous ideal $I$ and a monomial order on $R$. There exists a nonempty Zariski open set of $U_{gin}\subseteq \GL_d(k)$   so that $\init(g\cdot I)$ does not depend on $g$ for $g\in U_{gin}(I)$. The ideal $\init(g\cdot I)$ for $g\in U_{gin}(I)$ is termed the {\em generic initial ideal} of $I$ and denoted  $\gin(I)$.
\end{defn}

Whenever we refer to the generic initial ideal of a homogeneous ideal $I$ with respect to the graded reverse lexicographic order, we use the notation $\ginr(I)$ instead of $\gin(I)$.

A monomial ideal $J$ is termed {\em strongly stable} provided that whenever $i < j$ and $\mu$ is a monomial such that $\mu x_j \in J$, we have $\mu x_i \in J$ as well. A monomial ideal is called {\em Borel-fixed} if it is fixed by the action of the Borel subgroup of upper triangular matrices of $GL_d(K)$. If $\ch(k)=0$ the notions of Borel fixed and strongly stable are equivalent. In positive characteristic, strongly stable implies Borel fixed but not conversely. Borel fixed ideals which are not strongly stable appear in \Cref{ex:charp} and \Cref{ex:charp2}.

\begin{defn}
Recall that the Hilbert function of an $\N$-graded module $M$ over a graded $k$-algebra $R$ is the function $H(M):\N\to\N$ given by $H(M)(i)= \dim_k M_i$. 
The {\em Castelnuovo-Mumford  regularity} of $M$ is 
$\reg(M)=\max\{j-i \mid {\rm Tor}^R_i(M,k)_j\neq 0\}$.
 When $M$ has finite length this invariant can also be expressed as $\reg(M)=\max\{i \mid H(M)(i)\neq 0\}$.
\end{defn}

It is useful to recall several properties of generic initial ideals.

\begin{rem}[Properties of gins]
\label{prop:properties}

Let $I$ be a homogeneous ideal of $R=k[x_1,\dots, x_d]$.
\begin{enumerate}
\item \cite[Theorem 15.26]{Eisenbud} The Hilbert functions of $I$ and $\gin(I)$ agree: $H(I)=H(\gin(I))$.
\item \cite{BSreg} The regularity of $I$ and $\ginr(I)$ agree: $\reg(I)=\reg(\ginr(I))$
\item \cite[Corollary 15.25]{Eisenbud}
If $P \in \Ass(R/\ginr(I))$, then $P = (x_1,\ldots ,x_j)$ and  $\codim(I)\leq  j\leq d-\operatorname{depth}(I)$. % \cite[Proposition 4.2.9]{HH}
\item \cite{Galligo, BS} $\gin(I)$ is Borel fixed
\item \cite{Galligo} \cite[Proposition 4.2.4]{HH} If $\ch(k) = 0$ or $\ch(k)$ is larger than any exponent appearing in the monomial generators of $\gin(I)$, then $\gin(I)$ is strongly stable.
%\item \cite{EK} if $\ginr(I)$ is strongly stable, then $\reg(\ginr(I))$ is equal to the largest degree of a minimal generator of $\ginr(I)$.
\end{enumerate}
\end{rem}

\section{Sectional regularity}
\label{appendix}

In this section we develop the notion of sectional regularity introduced in \Cref{def:sreg}. The name is meant to suggest that this invariant captures the regularity of a general linear subspace section of a given quotient ring or module. It bears some similarities to the notion of partial regularity introduced by Trung in \cite{T2}, which in turn is equivalent to the $\ell$-regularity from the work  of Bayer--Charalambous--Popescu \cite{BCP}.  However we emphasize that these invariants are not the same as the sectional regularity introduced here. Sectional regularity has been previously considered from a geometric perspective in \cite{BLPS} and in subsequent work deriving from this source.

%\begin{defn}
%\label{def:sreg}
%Let $I$ be a homogeneous ideal of a polynomial ring $R=k[x_1,\ldots, x_d]$ and let $i$ be an integer $0\leq i\leq d$. Then there is a nonempty Zariski open set $U_{sreg}(I)\subset \A^{d^2}_k$ so that for each $\gamma=(g_{ij})\in U_{sreg}(I)$ setting $\ell_i=\sum_{j=1}^d g_{ij}x_j$ yields that $\reg(I + (\ell_{i+1},\ldots,\ell_d))$ is independent on the choice of $\gamma\in U_{sreg}(I)$. We define the $i$-th {\em sectional regularity} number of the ideal $I$ for any linear forms $\ell_i$ as above to be given by
%\[
%\sreg_i(I)=\reg(I + (\ell_{i+1},\ldots,\ell_d)).
%\]
%\end{defn}

A main goal of this section is to establish well-definedness of sectional regularity for a homogeneous ideal $I$ as introduced in \Cref{def:sreg}. This entails demonstrating the existence of  a nonempty Zariski open subset $U_{sreg}(I)\subset GL_d(k)$ so that for each $\gamma=(g_{ij})\in U_{sreg}(I)$, setting $\ell_i=\sum_{j=1}^d g_{ij}x_j$ yields the same value for the integer $\sreg_i(I)=\reg(I + (\ell_{i+1},\ldots,\ell_d))$. This will be shown in  \Cref{prop:regIvsgin(I)}. 

A useful tool in our work is the notion of almost regular sequence. This was introduced under the name of filter regular sequence in \cite{STT} and further developed in \cite{AH}. Here we follow the presentation in \cite[\S4.3.1]{HH}.

\begin{defn}
\label{def:almostreg}
Let $M$ be a finitely generated module over the polynomial ring $R=k[x_1,\dots, x_d]$. An element $\ell\in R$ 
is an {\em almost regular element }on  $M$ if the module $(0:_M \ell)=\{m\in M\mid \ell m=0\}$  is a finite dimensional $k$-vector space. 

 A sequence of elements $\ell_1, \ldots, \ell_s$  in $R$ is an {\em almost regular sequence} on $M$ if for each $i$ the element $\ell_{i+1}$ is almost regular on $M/\left(\ell_1, \ldots, \ell_i \right)$. 
\end{defn}
%The module $(0:_M \ell)$ in \Cref{def:almostreg} is the kernel of the multiplication homomorphism $\mu_\ell:M\to M, \mu_\ell(m)=\ell \cdot m$.
We will be interested in almost regular sequences consisting of linear forms. In the case of cyclic quotients by generic initial ideals, a canonical example of an almost regular sequence is given by the sequence of variables.

\begin{ex}[{\cite[Proposition 4.3.3]{HH}}]
Let $I$ be a homogeneous ideal over $R=k[x_1,\dots, x_d]$. Then $x_d,\ldots, x_1$ is almost regular on $R/\ginr(I)$.
\end{ex}

Since the above example is fundamental to this work, in order to keep the notation closely aligned with the example, we shall use regular sequences indexed descendingly, that is, we shall work with  elements $\ell_d, \ldots, \ell_1$ which form  an {\em almost regular sequence} on a module $M$. By \Cref{def:almostreg} this means that $\ell_{i-1}$ is almost regular on $M/\left(\ell_d, \ldots, \ell_i)\right)$ for each $2\leq i\leq d$. Whenever we cite results from the literature such as \Cref{lem:almostregopen} below we adapt the indexing to match this convention. 

A crucial fact about almost regular sequence is that, given a module $M$, a sufficiently general $k$-basis of $R_1$ yields an almost regular sequence of linear forms on $M$.

\begin{lem}[{\cite[Theorem 4.3.6]{HH}}]
\label{lem:almostregopen}
Let $M$ be a finitely generated graded module over $R=k[x_1,\dots, x_d]$. 
To each $\gamma = (g_{ij}) \in \GL_d(k)$ we associate the sequence $\ell_i=\sum_{j=1}^d g_{ij}x_j$ with $1\leq i\leq d$. Then there exists a nonempty Zariski open subset $U_{ar}(M) \subset \GL_d(k)$ such that $\ell_d, \ldots, \ell_1$ is almost regular on $M$ for all $\gamma\in U_{ar}$.

In particular, there exists a $k$-basis of $R_1$ which is an almost regular sequence on M.
\end{lem}

One can use almost regular sequences of linear forms to compute regularity.

\begin{lem}[{\cite[Proposition 20.20]{Eisenbud}}]
\label{lem:almostreg}
Let $M$ be a finitely generated graded module over the ring $R=k[x_1,\dots, x_d]$.
If $\ell\in R_1$ is almost regular on $M$ then 
\[ \reg(M) =\max\{\reg\left (M/(\ell)\right), \reg (0:_{M} \ell) \}.\]
\end{lem}

Together these facts allow us to relate the regularity of an ideal and that of its generic initial ideal modulo appropriate almost regular sequences. 

\begin{prop}
\label{prop:regIvsgin(I)}
Let $I$ be a homogeneous ideal of a polynomial ring $R=k[x_1,\dots, x_d]$. Then there exists a nonempty Zariski open subset $U_{sreg}(I)\subset GL_d(k)$ so that for each $\gamma=(g_{ij})$ in $U_{sreg}(I)$ setting $\ell_i=\sum_{j=1}^d g_{ij}x_j$ yields that $\ell_d,\ldots, \ell_1$ is an almost regular sequence on $R/I$. Moreover, for all $0\leq i\leq d-1$, one has
\[
H \left(I+(\ell_{i+1}, \ldots, \ell_d)\right) =H\left( \ginr(I)+(x_{i+1}, \ldots, x_d)\right).
\]
\[
\reg\left(I+(\ell_{i+1}, \ldots, \ell_d)\right) =\reg\left( \ginr(I)+(x_{i+1}, \ldots, x_d)\right).
\]
\end{prop}
\begin{proof}
We set $U_{gin}^{-1}(I)=\{\gamma \in GL_d(k) \mid \gamma^{-1}\in U_{gin}(I)\}$, where $U_{gin}(I)$ is the Zariski open set in \Cref{def:gin}. Since $U_{gin}(I)$ is an open set, so is $U_{gin}^{-1}(I)$; see \cite[Lemma 4.3.8]{HH}. Set $U_{sreg}(I)=U_{gin}^{-1}(I)\cap U_{ar}(R/I)$, where $U_{ar}$ is the Zariski open set in \Cref{lem:almostregopen}. Since $\gamma\in U_{ar}(R/I)$, the sequence $\ell_d, \ldots, \ell_1$ is almost regular on $R/I$.

We proceed to establish the claimed identities. The  isomorphisms
\begin{eqnarray*}
I+(\ell_{i+1},\cdots,\ell_d) \cong \gamma^{-1}(I+(\ell_{i+1},\cdots,\ell_d))=\gamma^{-1}(I)+(x_{i+1},\cdots,x_d)
\end{eqnarray*}
give the first equality in the sequence
\begin{eqnarray*}
H(I+(\ell_{i+1},\cdots,\ell_d)) &=&H(\gamma^{-1}(I)+(x_{i+1},\cdots,x_d))\\
&=&H(\init(\gamma^{-1}(I)+(x_{i+1},\cdots,x_d))) \quad \text{ by } \Cref{prop:properties} \text{ (1)}\\
&=&H(\init(\gamma^{-1}(I))+(x_{i+1},\cdots,x_d))  \\%\quad \text{ by } \cite[Lemma 4.3.7]{HH}\\
&=&H(\ginr(I)+(x_{i+1},\cdots,x_d)) \qquad  \text{ since } \gamma^{-1}\in U_{gin}(I).
\end{eqnarray*}

As a byproduct of this equality we obtain another useful identity. Set 
\[M^{(i)}=R/(I+(\ell_{i+1},\cdots,\ell_d)), \quad N^{(i)}=R/(\ginr(I)+(x_{i+1},\cdots,x_d))\] and note that $H(M^{(i)})=H(N^{(i)})$ as shown above.
Then  the exact sequences 
\[
0\to (0:_{M^{(i)}} \ell_{i})(-1) \to M^{(i)}(-1) \xrightarrow{\cdot \ell_i} M^{(i)} \to M^{(i-1)}\to 0
\]
\[
0\to (0:_{N^{(i)}} x_{i})(-1) \to N^{(i)}(-1) \xrightarrow{\cdot x_i} N^{(i)} \to N^{(i-1)}\to 0
\]
and the equalities $H(M^{(i)})=H(N^{(i)})$ yield $H(0:_{M^{(i)}} \ell_{i})=H(0:_{N^{(i)}} x_{i})$ for $1\leq i\leq d$. Since the regularity of the finite length modules $(0:_{M^{(i)} } \ell_{i})$ and $(0:_{N^{(i)} } x_{i})$ is determined by their Hilbert function it follows that $\reg(0:_{M^{(i)}} \ell_{i})=\reg(0:_{N^{(i)}} x_{i})$.

We now establish the claim regarding regularity by induction on $i$. The case $i=0$ holds true because $M^{(0)}=N^{(0)}=k$. \Cref{lem:almostreg} (1) cast  in terms of our current notation, where $\ell_i$ is almost regular on $M^{(i)}$ and $M^{(i-1)}=M^{(i)}/(\ell_i)$, yields
\begin{equation}
\label{eq:1}
\reg(M^{(i)})=\max\left \{\reg(M^{(i-1)}), \reg(0:_{M^{(i)}}\ell_{i}) \right \}
\end{equation}
and similarly 
\[
\reg(N^{(i)})=\max\left \{\reg(N^{(i-1)}), \reg(0:_{N^{(i)}}x_i)\right \}.
\]
The inductive hypothesis guarantees $\reg(M^{(i-1)})=\reg(N^{(i-1)})$ and prior considerations yield  $\reg(0:_{M^{(i)}}\ell_{i})=\reg(0:_{N^{(i)}}x_{i})$, therefore the above identities imply $\reg(M^{(i)})=\reg(N^{(i)})$. From these equalities the last claim follows.
\end{proof}

We now present an alternate formula for the sectional regularity. This can be formulated in terms of generic annihilator numbers, which we now define.

\begin{defn}
\label{def:alphaij}
Let $I$ be a homogeneous ideal of a polynomial ring $R$ with $\dim(R)=d$ and let $\ell_d, \ldots, \ell_1$ be an almost regular sequence for $R/I$. Let $M^{(j)}=R/(I + (\ell_{j+1},\ldots,\ell_d))$   and define
\[
\alpha_{ij}(R/I)=\dim_k(0:_{M^{(d-i)}} \ell_{d-i})_j  \text{ for } 0\leq i\leq d-1
\]
to be the {\em generic annihilator numbers} of $R/I$. We do not specify values for annihilator numbers if $i=d$; although these appear in the literature (see \cite{AH}) they are not relevant here.
\end{defn}

\begin{cor}
\label{cor:sregmax}
Continuing with the notation in \Cref{def:alphaij}
 one has 
\[
\sreg_i(I)=\max\{\reg(0:_{M^{(t)}} \ell_t) \mid t \leq  i\}+1=\max\{j \mid \alpha_{sj}(R/I)\neq 0,s \geq d-i\}+1.
\]
Consequently the sequence of sectional regularity numbers is nondecreasing, that is, 
\[
 \sreg_1(I) \leq  \sreg_2(I) \leq  \ldots \leq \sreg_d(I).
\]
\end{cor}
\begin{proof}
Applying equation \eqref{eq:1} repeatedly yields
\[
\reg(M^{(i)})=\max \left\{\reg\left(M^{(0)}\right), \reg\left(0:_{M^{(t)}} \ell_t\right) \mid t \leq  i \right \}.
\]
Since $M^{(0)}=k$, and $\sreg_i(I)=\reg(M^{(i)})+1$ by definition, the first claimed equality follows.
The second is obtained by noting that
$\reg\left(0:_{M^{(d-t)}} \ell_{d-t}\right)=\max\{j\mid \alpha_{tj}(R/I)\neq 0\}$.
Either one of the descriptions for $\sreg_i(I)$ obtained above yields that this sequence is non-decreasing.
\end{proof}

We now introduce a new definition of coextremal annihilator numbers.

\begin{defn}
\label{def:coextremal}
Let $I$ be a homogeneous ideal of a polynomial ring $R$. A generic annihilator number $\alpha_{ij}$ as defined in \Cref{def:alphaij} is called a {\em coextremal generic annihilator number} if it satisfies $\alpha_{ij}(R/I)\neq 0$ and $\alpha_{st}(R/I)=0$  for all pairs  $(s,t)\neq (i,j)$ with  $s \geq i,t \geq j$.
\end{defn}

The following result specifies the meaning of the coextremal annihilator numbers. 

\begin{cor}
\label{cor:coextremal}
Let $I$ be a homogeneous ideal of a polynomial ring $R$ with $\dim(R)=d$. 

If $\alpha_{d-i,j}(R/I)$ is a coextremal generic annihilator number then $\sreg_{i}(I)=j+1$. 
Moreover, if $\sreg_{i-1}(I)<\sreg_i(I)=j+1$ then $\alpha_{d-i,j}(R/I)$ is a coextremal annihilator number. 

Consequently the coextremal annihilator numbers control both the jumps and the values attained by the sequence 
$
 \sreg_1(I) \leq  \sreg_2(I) \leq  \ldots \leq \sreg_d(I).
$
\end{cor}
\begin{proof}
Suppose $\alpha_{d-i,j}(R/I)$ is a coextremal generic annihilator number. Then from \Cref{def:coextremal} it follows that $\max\{j \mid \alpha_{st}(R/I)\neq 0, s \geq d-i\}=j$. Combining this with \Cref{cor:sregmax} yields  $\sreg_{i}(I)=j+1$.

Suppose now that we have $\sreg_{i-1}(I)<\sreg_i(I)=j+1$. This leads to the inequality
$$\sreg_{i-1}(I)=\max\{\reg(0:_{M^{(u)}} \ell_u) \mid u \leq  i-1\}+1< \max\{\reg(0:_{M^{(u)}} \ell_u) \mid u \leq  i\}+1=\sreg_{i}(I),$$ 
which forces $\reg(0:_{M^{(i)}} \ell_i)> \reg(0:_{M^{(u)}} \ell_u)$ for each $u<i$. Casting this in terms of generic annihilator numbers yields
\[
j=\max\{t \mid \alpha_{d-i,t}(R/I)\neq 0 \}>\max\{ t\mid \alpha_{d-u,t}(R/I)\neq 0 \text{ for some } u<i\},
\]
which implies $\alpha_{st}(R/I)=0$ for $(s,t)\neq (d-i,j), s\geq d-i, t\geq j$, that is, $\alpha_{d-i,j}(R/I)$ is coextremal.

Finally, to justify how the coextremal annihilator numbers control the sequence of sectional regularities it suffices to see that the distinct values in the sequence are given by
\[
\{ \sreg_1(I),  \sreg_2(I), \ldots, \sreg_d(I)\}=\{j +1 \mid \alpha_{ij}(R/I) \text{ is extremal for some } i\}
\]
and the jumps are given by $\sreg_{i-1}(I)<\sreg_i(I)$ if and only if there exists $j$ such that $\alpha_{d-i,j}(R/I)$ is a coextremal annihilator number. Due to its monotonicity, the sequence of sectional regularities is completely determined by these two pieces of information.
\end{proof}

Finally, we discuss an alternate approach to almost regular sequences that uses generic linear forms, that is, linear forms whose coefficients are variables in an extension field of the coefficient field of $R$. Generic linear forms have the advantage that the {\em same} set of such forms gives an almost regular sequence on {\em any} finitely generated $R$-module. Therefore this approach is better suited to arguments where $R$-modules vary in families. We make use of such arguments in \cref{s:linear}.

\begin{notation}
\label{not:generic}
Let $u_{ij}$ be distinct variables with $1\leq i,j\leq d$, set $F$ to be the fraction field of $k[u_{ij}]$, and $R'=F\otimes_k R=F[x_1,\ldots, x_d]$. Additionally, for any $R$-module $M$ set $M'=M\otimes_R R'$ and for $1\leq i\leq d$ set $L_i=\sum_{j=1}^d u_{ij}x_j$.
\end{notation}

 The following is the analogue of \Cref{prop:regIvsgin(I)} for generic forms.

\begin{prop}
\label{prop:regIvsgin(I)generic}
Let $I$ be a homogeneous ideal of a polynomial ring $R=k[x_1,\dots, x_d]$. According to \Cref{not:generic}, let $I'=IR'$ and set
$L_i=\sum_{j=1}^d u_{ij}x_j$.  Then $L_d, \ldots, L_1$ is an almost regular sequence on $R'/I'$ and the following identities hold for each $0\leq i\leq d$
\[
H\left(I'+(L_{i+1}, \ldots, L_d)\right) =H\left( \gin(I)+(x_{i+1}, \ldots, x_d)\right).
\]
\[
\reg\left(I'+(L_{i+1}, \ldots, L_d)\right) =\reg\left( \ginr(I)+(x_{i+1}, \ldots, x_d)\right).
\]
\end{prop}
\begin{proof}
 It is shown e.~g.~in \cite[Proposition 2.1]{AH} that for any $R$-module $M$, the sequence $L_d,\ldots, L_1$ and any permutation thereof is almost regular on $M'$. 
 
 Set $\phi:R'\to R'$ be the automorphism sending $L_i \mapsto x_i$. To be precise, consider the matrix $(v_{ij})=(u_{ij})^{-1}$ with entries in $F$. Then $\phi$ is defined by $\phi(x_i)=\sum_{j=1}^d v_{ij}x_j$. We claim that $\init(\phi(I))=\gin(I)R'$.
 
 Suppose $\init(\phi(I'))=(m_1, \ldots, m_s)$ with $m_i=\init(f_i)$ for some polynomials $f_i\in I'$ and let $c_i$ be the coefficient of $m_i$ in $f_i$. Then $c_i\in F$ are rational functions of $u_{ij}$. Consider the Zariski open set $U'\subset \A^{d^2}_k$  where none of the denominators and numerators of the functions $c_i$  vanish and set $U=U_{gin}^{-1}(I)\cap U'$.  
 Let $\gamma=(g_{ij})\in U$ and $e_\gamma:R'\to R$ the evaluation homomorphism that maps $u_{ij}\mapsto g_{ij}$. Then $e_\gamma\circ \phi(f)= \gamma^{-1}\cdot f$ for any $f\in R$. The membership
\[
m_i=\init(e_\gamma(f_i))\in \init(e_\gamma\circ \phi(I'))=\init(\gamma^{-1}(I))R'=\ginr(I)R'
\]
shows there is a containment $\init(\phi(I)) \subseteq \ginr(I)R'$. Since there are equalities 
\[H(\init(\phi(I)))=H(\phi(I'))=H(I')=H(I)=H(\ginr(I))= H (\ginr(I)R'),\]
 the preceding containment is in fact an equality.
 
 Set $M^{(i)}=R/(I'+(L_{i+1},\cdots,L_d))$ and $N^{(i)}=R/(\ginr(I)R'+(x_{i+1},\cdots,x_d))$.
 Then 
 \begin{eqnarray*}
 H(M^{(i)}) &=& H(\phi (M^{(i)})) = H\left( R/(\phi(I')+(x_{i+1},\cdots,x_d))\right)\\
 &=& H\left( R/\init(\phi(I')+(x_{i+1},\cdots,x_d))\right)\\
 &=& H\left( R/\init(\phi(I'))+(x_{i+1},\cdots,x_d)\right)\\
 &=& H\left( R/(\ginr(I)R'+(x_{i+1},\cdots,x_d)\right)\\
 &=& H(N^{(i)}).
 \end{eqnarray*}
 Now the proof of \Cref{prop:regIvsgin(I)} can be repeated verbatim to deduce $\reg(M^{(i)})=\reg(N^{(i)})$ for $0\leq i\leq d$, as desired. 
\end{proof}

\section{Equivalent invariants}
\label{s:2}

In this section we offer multiple equivalent characterizations for the notion of sectional regularity, establishing in particular   the equivalence of sectional regularity and axial constants. Our main result is the following.

\begin{thm}
\label{thm:equiv}
Let $I$ be a homogeneous ideal of a polynomial ring $R=k[x_1, \ldots, x_d]$ with $k$ an infinite  field of characteristic zero or larger than any exponent appearing in the generators of the reverse lexicographic generic initial ideal $\ginr(I)$. The following are equal:
\begin{enumerate}
\item the $i$-th sectional regularity number $\sreg_i(I)$ (see \Cref{def:sreg}) 
\item $\sreg_i(\ginr(I))=\reg(\ginr(I) + (x_{i+1},\ldots,x_d))$
\item the largest $i$-th partial degree of a minimal generator of $\ginr(I)$ (see \Cref{def:degi}) 
\item the $i$-th axial constant $a_i(I)$ (see \Cref{def:axial}), provided $a_i(I)<\infty$
\item $\red_{d-i}(R/I)+1$, where $\red_{d-i}(R/I)$ is an $s$-reduction number for $R/I$ (see \Cref{def:sred}), provided $\red_{d-i}(R/I)<\infty$.
\end{enumerate}

Moreover, the quantities listed in (1) and (2)  and, separately,  those listed in (4) and (5) are equal without any restrictions on characteristic.

\end{thm}
\begin{proof}
This result collects together  \Cref{prop:regIvsgin(I)}, which shows (1)=(2), \Cref{prop:sreg=alphial}, which shows (2)=(3), 
 \Cref{prop:axial=sectional reg}, which shows (2)=(4), and \Cref{prop:axial=red}, which shows (4)=(5).
The hypothesis on the characteristic  ensures that $\ginr(I)$ is strongly stable; see \cite[Proposition 2.7]{BS} or \cite[Proposition 4.2.4]{HH}. This technical condition is needed in the proofs of \Cref{prop:axial=sectional reg} and \Cref{prop:sreg=alphial},  but not for \Cref{prop:regIvsgin(I)} and \Cref{prop:axial=red}. This justifies the last claim regarding those equalities that are independent of the characteristic of the base field.
\end{proof}

The hypothesis on the characteristic of the base field in \Cref{thm:equiv} is needed, as illustrated by the next examples. In these examples we employ gradings of the polynomial ring termed partial gradings, denoted $\deg_i$, and given by $\deg_i(x_j)=1$ if $j\leq i$ and $\deg_i(x_j)=0$ for $i<j$.

\begin{ex}
\label{ex:charp}
Consider the ideal $I=(x_1^p, \ldots, x_d^p)\subset R=k[x_1,\ldots, x_d]$, where $\ch(k)=p$. 
It is easy to see that $I=\ginr(I)$ and that this ideal is not strongly stable, however it is Borel-fixed. Moreover, one has 
\begin{itemize}
\item $a_i(I)=p$ for each $1\leq i\leq d$
\item $\sreg_I(I)=\reg(\ginr(I) + (x_{i+1},\ldots,x_d))=\reg(x_1^p, \ldots, x_i^p, x_{i+1}, \ldots, x_d)=i(p-1)+1$
\item  the largest $i$-th partial degree of a minimal generator of $\ginr(I)$ is $p$
\item  the $s$-reduction numbers are $\red_s(R/I)=(p-1)(d-s)$ cf.~\cite[Example 1.3]{HT}
\end{itemize}
Thus here the axial constants agree with the the largest $i$-th partial degree of a minimal generator of $\ginr(I)$ and the sectional regularity numbers satisfy $\sreg(I)=\red_{d-i}(R/I)+1$. However the axial constant and the sectional regularity disagree when $i>1$.
\end{ex}

In the previous example the connection between axial constants and the largest $i$-th partial degree of a minimal generator of $\ginr(I)$ appears more robust than we prove in \Cref{prop:axial=alphial}. However the next example shows that our result is in fact the best possible. 

\begin{ex} 
\label{ex:charp2}
Let $p$ be a prime integer and set $n=p$ if $p$ is odd and $n=4$ if $p=2$.
Consider the ideal $I=(x_1^n, x_2^n, x_3^n\ldots, x_d^n)+(x_1^{n-1}x_2^2)\subset R=k[x_1,\ldots, x_d]$, where $\ch(k)=p$. 
Then $I=\ginr(I)$, the axial constants are $a_i(I)=n$ for $1\leq i\leq d$ but  the largest partial degree of a minimal generator of $\ginr(I)$ with respect to the grading $\deg_i$  is $n+1$, provided $i\geq 2$.  Moreover the $i$-th sectional regularity is $\sreg_1(I)=n$ and $\sreg_i(I)=(n-1)i$ for $2\leq i\leq d$. Therefore in this example the quantities listed in (1), (3), and (4) of \Cref{thm:equiv} are pairwise different for $i\geq 2$.\end{ex}

We now proceed to show the equivalences between the families of constants in \Cref{thm:equiv}.

\subsection{Axial constants}
To study the axial constants, we first establish some useful properties for the sequence formed by these invariants.
\begin{rem}
\label{rem:monotoneaxial}
From \Cref{thm:equiv} and \Cref {cor:sregmax} one deduces that if $\ginr(I)$ is a strongly stable ideal, the axial constants form a non-decreasing sequence
$
a_1(I)\leq a_2(I) \leq \cdots \leq a_d(I).
$
%This property is also easy to deduce directly from the strongly stable property of $\ginr(I)$.
\end{rem}
 However the sequence of axial constants  is usually unbounded. We proceed by delineating the range in which the axial constants are finite.

\begin{lem}
\label{lem:finiteaxial}
Let $I$ be a homogeneous ideal and assume $\ginr(I)$ is strongly stable. Then $a_i\in \N$ if and only if $1\leq i\leq \codim(I)$.
\end{lem}
\begin{proof}
According to \Cref{def:axial}, $a_i\in \N$ if and only  $x_i\in \sqrt{\ginr(I)}$. Since $\ginr(I)$ is strongly stable, the associated primes of $\ginr(I)$ are of the form $(x_1,\ldots, x_j)$ with $\codim I\leq j$; see \Cref{prop:properties}(3). Therefore we have
\[
\sqrt{\ginr(I)}=\bigcap_{P\in \Ass(\ginr(I))} P=(x_1,\ldots, x_{\codim(I)})
\]
and consequently $x_i\in \sqrt{\ginr(I)}$ if and only if $1\leq i\leq \codim(I)$. The conclusion now follows.
\end{proof}

\begin{lem}
\label{lem:addvars}
Suppose $J$ is a strongly stable monomial ideal of a polynomial ring $R$. Then the monomial $x_i^j$ is in $J$ if and only if $(J + (x_{i+1},\ldots,x_d))_j = R_j.$
\end{lem}
\begin{proof}
%To start, suppose $\gin(I)$ is minimally generated by $m_1,\ldots,m_s.$ Then we have that
%$$L = \gin(I) + (x_{i+1},\ldots,x_d) = (m_1,\ldots,m_s,x_{i+1},\ldots,x_d).$$

The ``if" direction is clear. % If $(J+ (x_{i+1},\ldots,x_d))_j = R_j$ then we see that $$x_i^j\in R_j\in J + (x_{i+1},\ldots,x_d),$$ from which it follows that $x_i^j \in J$ as desired.
For the ``only if" direction,  suppose $x_i^{j}\in J$. We show that $L=J + (x_{i+1},\ldots,x_d)$ contains all monomials of degree $j$ from which the desired statement follows. Let $m$ be a monomial of degree $j.$ If $x_k$ divides $m$ for some $k>i$ then it immediately follows that $m\in L$. Otherwise $m$ is expressed in terms of $x_1, \ldots, x_i$, and as $x_i^j\in J$ one can apply the strongly stable exchange property repeatedly to $x_i^j$ to obtain that $m\in J\subseteq L$.
\end{proof}

Now we are ready to give a homological characterization of the axial constants. The alert reader will realize that the next result justifies the equality $a_i(I) =\sreg_i(I)$ for $1\leq i\leq \codim(I)$ advertised in  \Cref{thm:equiv}.

\begin{prop}
\label{prop:axial=sectional reg}
Let $I$ be a homogeneous ideal of a polynomial ring and assume that $\ginr(I)$ is strongly stable.  Then
\begin{enumerate}
\item
$a_i(I) = \reg(\ginr(I) + (x_{i+1},\ldots,x_d))$ provided $a_i(I)<\infty$.
\item
there exists a nonempty Zariski open subset $U\subset GL_d(k)$ so that for each $\gamma=(g_{ij})\in U$ setting $\ell_i=\sum_{j=1}^d g_{ij}x_j$ yields 
$a_i(I) = \reg(I + (\ell_{i+1},\ldots,\ell_d))$.
\item setting $F = {\rm Frac}(k[u_{ij}])$, $R' = F[x_1,\ldots,x_d]$ and $I' = IR'$ and considering the generic linear forms $L_i = \sum_{j=1}^{d} u_{ij} x_j$ yields
$a_i(I) = \reg(I' + (L_{i+1},\ldots,L_d))$.
\end{enumerate}
\end{prop}
\begin{proof}
(1) By \Cref{lem:finiteaxial}, $a_i(I)$ is finite if and only if $1\le i\le \codim(I)$. Under this assumption,  $R/\left(\ginr(I) + (x_{i+1},\ldots,x_d)\right)$ is a zero-dimensional ring,  which implies that 
\[
 \reg(\gin(I) + (x_{i+1},\ldots,x_d))=\min\{j\mid \gin(I) + (x_{i+1},\ldots,x_d))_j = R_j\}.
 \]
 The desired conclusion now follows by means of \Cref{lem:addvars}.
 
Assertions (2) and (3) follow from (1) by means of \Cref{prop:regIvsgin(I)} and \Cref{prop:regIvsgin(I)generic}.
\end{proof}

\begin{rem}
\label{remfirstlast}
The first axial constant is easily seen from \Cref{def:axial} to be the least degree of an element of $I$. Assuming $\ginr(I)$ is strongly stable and $\dim(R/I)=\operatorname{depth}(R/I)$, we now compute the largest finite axial constant, $a_c(I)$, where $c=\codim(I)$. By \Cref{prop:axial=sectional reg} 
\[
a_c(I)=\reg(\ginr(I) + (x_{c+1},\ldots,x_d))=\reg(\ginr(I))=\reg(I),
\]
where the middle equality holds because  \Cref{prop:properties}(4) ensures that $x_{c+1},\ldots,x_d$ is a regular sequence on $R/\ginr(I)$, and the last equality follows from \Cref{prop:properties}(3). 
\end{rem}

Given the close relationship between sectional regularity numbers and generic annihilator numbers observed in section \ref{appendix}, we are now able to characterize axial constants in terms of the coextremal annihilator numbers.

\begin{cor}
\label{cor:axialvalues}
Let $I$ be a homogeneous ideal of a polynomial ring $R$ with $\dim(R)=d$ and $\codim(I)=c$. Then, setting $a_0(I)=-\infty$ we have $a_{i-1}(I)<a_i(I)=j+1<\infty$  if and only if  $\alpha_{d-i,j}(R/I)$ is coextremal and moreover
\[
\{a_1(I), a_2(I), \ldots, a_c(I)\}=\{j\mid \alpha_{i,j}(R/I) \text{ is coextremal for some } i\}.
\]
 Thus the coextremal generic annihilator numbers control both the jumps and the values attained by the sequence 
$a_1(I) \leq  a_2(I) \leq  \ldots \leq a_c(I)$.
\end{cor}
\begin{proof}
It follows from \Cref{prop:axial=sectional reg} that $a_i(I)=\sreg_i(I)$. Therefore this result transcribes the findings of \Cref{cor:coextremal} in the case when the two invariants agree.
\end{proof}

%Given the meaning attached to the least and greatest finite axial constants, \Cref{thm:equiv} can be interpreted as providing a range of  new invariants attached to a homogeneous ideal which generalize the usual measures of structural and homological complexity for homogeneous ideals. 

\subsection{Partial degrees}

It is well known that the regularity of a strongly stable ideal is given by the largest degree of a minimal generator. In this section we generalize this framework to capture the sectional regularity of a strongly stable ideal in terms of the largest partial degree of a minimal generator. We now proceed to define the notion of partial degree. 

\begin{defn}
\label{def:degi}
 Given a monomial $\mu=x_1^{a_1}\ldots x_d^{a_d}$ and an index $1\le i\le d$ we define the $i$-th {\em partial degree} of $\mu$ to be $\deg_{i}(\mu)=\sum_{j\le i} a_j$. This yields a grading on the polynomial ring $k[x_1,\ldots, x_d]$, which we also denote $\deg_i$.
\end{defn}

Note that $\deg_d$ is the standard grading. % An ideal that is homogeneous with respect to some $\deg_i$ need not remain homogeneous with respect to $\deg_j$ for $j\neq i$. Nevertheless, 
Monomial ideals are homogeneous with respect to $\deg_i$ for each $i$.

%We shall be able to get considerable mileage out of the following observation. 
%
%\begin{rem}
%Set $\fm=(x_1,\ldots, x_d)$. For each $i$, the finite dimensional graded vector space $I/\fm I$ uniquely determines a set of positive integers, namely the set of degrees of its elements.
%\end{rem}
%
%In particular, this observation allows us to define the alphial constants of an ideal.

\begin{defn}
\label{def:alphial}
Let $J$ be a monomial ideal of a polynomial ring $R=k[x_1,\ldots, x_d]$. For each $1\leq i\leq d$ we define 
$ \omega_i(J)=\max\{ \deg_i(\mu) \mid \mu \text{ a minimal monomial generator of } J\}$.

%\[
%\omega_i(I)=\max\{ \deg_i(\overline{f}) \mid \overline{f}\in I/(x_1,\ldots, x_d)I\}.
%\]
%If $I$ is homogeneous with respect to $\deg_i$ then $\omega_i(I)$ is the largest degree of a generator in any minimal homogeneous generating set for $I$.
\end{defn}

It is well-known that the regularity of strongly stable ideals is equal to the largest among the degrees of their minimal generators. Below we establish a version of this fact for sectional regularity.

\begin{prop}
\label{prop:sreg=alphial}
Let $J$ be a strongly stable monomial ideal of a polynomial ring. Then for $1\leq i \leq d$ one has 
$\reg(J+ (x_{i+1},\ldots,x_d))= \omega_i(J)$.
\end{prop}
\begin{proof}
Let $L$ be the ideal generated by  the minimal monomial generators of $J$ that are not divisible by $x_{i+1},\ldots,x_d$. Since $J$ is assumed strongly stable, so is $L$. Therefore by means of the Eliahou-Kervaire resolution \cite{EK} there is an equality $\reg(L)=\omega_i(L)$. 

The identity
\[
J+ (x_{i+1},\ldots,x_d))= L+(x_{i+1},\ldots,x_d))
\]
yields 
\begin{eqnarray*}
\reg\left(J + (x_{i+1},\ldots,x_d)) \right)&=& \reg\left(L+(x_{i+1},\ldots,x_d))\right)\\
&=& \reg\left(L\right) =\omega_i(L),
\end{eqnarray*}
where the second equality follows because $x_{i+1},\ldots,x_d$ form a regular sequence on $L$.
It remains to show that $\omega_i(J)=\omega_i(L)$ in order to complete the proof.
 
Suppose $\mu=x_1^{a_1},\ldots,x_i^{a_i} \mu'$ is a monomial minimal generator of $J$ so that $\mu' \not \in  (x_{1},\ldots,x_i)$ and  $\deg_i(\mu)=\omega_i(J)$. We show $\mu\in L$, equivalently, $\deg(\mu')=0$.
Suppose for the sake of contradiction that $\deg(\mu')>0$. Then, applying the strongly stable property, yields
$$x_1^{a_1},\ldots,x_i^{a_i+1} \mu'/x_j\in J$$
 for some $j>i$ so that $x_j\mid \mu'$. The displayed monomial is a multiple of a minimal generator
$$\mu''=x_1^{b_1},\ldots,x_i^{b_i} \mu''' \in J,$$
 where $b_k\le a_k$ for $1\leq k\leq i-1$, $b_{i }\leq a_i+1$, and $\mu'''\mid \mu'/x_j.$ If $b_i \le a_i$ then $\mu$ is a multiple of $\mu''$ contradicting minimality of $\mu$. Thus $b_i = a_i+1$ and  since $\mu$ has the largest $\deg_i$ among the minimal generators of $J$ we must have $b_1+\ldots b_{i-1}<a_1+\ldots+a_{i-1}.$ So, at least one of the exponents $b_k$ satisfies $b_k<a_k$. Applying the strongly stable property to $\mu''$ gives that 
$$x_1^{b_1}\ldots x_k^{b_k+1} \ldots x_i^{a_i} \mu'''\in J$$ Since $\mu$ is a multiple of the above monomial, this contradicts the minimality of $\mu$ and finishes the proof. 
\end{proof}

As a consequence of \Cref{prop:sreg=alphial} and  \Cref{prop:axial=sectional reg} we can now give a new characterization of the axial constants in terms of partial degrees.

\begin{cor}
\label{prop:axial=alphial}
%If $1\le i\le d$ and $a_i(I)$ is defined then $a_i(I)$ is the largest $\deg_i$ of a monomial in the minimal set of generators for $\gin(I)$.
Let $I$ be a homogeneous ideal of a polynomial ring and assume that $\ginr(I)$ is strongly stable.  Then for each $i$ such that $a_i(I)<\infty$ we have $a_i(I)=\omega_i(\ginr(I))$. 
\end{cor}

\subsection{$s$-reduction numbers}

The relationship between regularity and reduction numbers is well known starting with pioneering work in \cite{Vasconcelos}. The $s$-reduction numbers were introduced in \cite{NR}. The fact that reduction numbers remain unchanged under taking generic initial ideals was shown in \cite{T2} and their interpretation for Borel-fixed ideals recalled in \Cref{lem:HT} was introduced in \cite{HT}. 

\begin{defn}
\label{def:sred}
Let $A$ be a standard graded algebra over an infinite field $k$. An ideal $\fq = (\ell_1,\ldots ,\ell_s)$, where $\ell_i$ are linear forms of $A$, is called an {\em $s$-reduction} of $A$ (or of the homogeneous maximal ideal of $A$) if $\fq_t = A_t$ for $t\gg 0$. The {\em reduction number} of $A$ with respect to $\fq$ is 
$ 
r_\fq(A)=\min\{t \mid \fq_{t+1} = A_{t+1}\}.
$
The {\em $s$-reduction number} of $A$ is defined as
\[
r_s(A) = \min\{r_\fq(A) \mid \fq = (\ell_1, \ldots, \ell_s) \text{ is a reduction of }A\}.
\]
\end{defn}
 
 It is not hard to see that any $s$-reduction of $A$ satisfies $s\geq \dim(A)$. If $s=\dim(A)$ then $r(A):=r_s(A)$ is called the reduction number of $A$. If $s<\dim(A)$ then we set $r_s(A)=\infty$ by convention.

The following characterization for the $s$-reduction numbers of strongly stable monomial ideals can be found in \cite{HT}.

\begin{lem}[{\cite[Corollary 1.4]{HT}}]
\label{lem:HT}
Let  $J$ be a Borel fixed monomial ideal of a polynomial ring of dimension $d$. For any $s \geq \dim(R/J)$ we have 
$
r_s(R/J) = \min\{t \mid x_{d-s}^{t+1} \in I \}
$.
\end{lem}

Now we can give a characterization of the axial constants in terms of $s$-reduction numbers.

\begin{prop}
\label{prop:axial=red}
Let $I$ be a homogeneous ideal of a polynomial ring $R=k[x_1, \ldots, x_d]$. For any $0\leq i\leq d$ we have 
$
a_i(I)=r_{d-i}(R/I)+1.
$
\end{prop}
\begin{proof}
 From \Cref{lem:finiteaxial} and the above considerations it follows that $a_i(I)=\infty$ if and only if $ i>\codim(I)$  if and only if $d- i<\dim(R/I)$ if and only if $r_{d-i}(R/I)=\infty$.

Now assume $ i \leq \codim(I)$, whence $d- i \geq \dim(R/I)$. Since $\ginr(I)$ is a strongly stable monomial ideal, \Cref{lem:HT} gives 
\[
r_{d-i}(R/\ginr(I)) = \min\{t \mid x_{i}^{t+1} \in \ginr(I) \} = a_i(I)-1.
\]
It was shown in \cite{T2} that $r_s(R/I) = r_s(R/\ginr(I))$, which finishes the proof.
\end{proof}

\section{Growth of the equivalent invariants for powers of ideals}
\label{s:linear}

Applying a fixed algebraic invariant $\theta$ to a family of homogeneous ideals $\cF=\{I_n\}_{n\geq 1}$ gives rise to 
a function $n\mapsto \theta(I_n)$. A natural endeavor is to study the growth of such functions.

%\subsection{Linear growth}
In this section we fix a homogeneous ideal $I$ of a polynomial ring and focus on the sequence of powers of $I$, namely $\cF=\{I^n\}_{n\geq 1}$. It is well known that the regularity of powers of $I$ is given by an eventually linear function $\reg(I^n)=dn+e$ for $n\gg0$ by the celebrated work of Cutkosky--Herzog--Trung \cite{CHT} and Kodyialam \cite{Kod}.
We use this to show below that applying each of the invariants listed in \Cref{thm:equiv} to  $\cF$ also yields an eventually linear function.

\begin{thm}
\label{thm:linear}
Let $I$ be a homogeneous ideal of a polynomial ring $R=k[x_1, \ldots, x_d]$ with $k$ an infinite  field of characteristic zero or larger than any exponent appearing in the generators of the reverse lexicographic generic initial ideal $\ginr(I)$. Fix an integer $0\leq i\leq d$. Then there exist nonnegative integers $d,e\in\N$ such that $ \sreg_i(I^n)=dn+e$  for $n\gg 0.$
%\begin{enumerate}
%\item $f(n)= \sreg_i(I^n)$.
%\item $f(n)=\reg(\ginr(I^n) + (x_{i+1},\ldots,x_d))$ for a fixed integer $0\leq i\leq d$,
%\item $f(n)=a_i(I^n)$ for a fixed integer $0\leq i\leq \codim(I)$,
%\item $f(n)=$the largest $i$-th partial degree of a minimal generator of $\ginr(I^n)$. 
%\item $f(n)=\red_{d-i}(R/I^n)$ for a fixed integer $0\leq i\leq \codim(I)$.
%\end{enumerate}
%Then there exist nonnegative integers $d,e\in\N$ such that $f(n)=dn+e$  for $n\gg 0.$
\end{thm}
\begin{proof}
By \Cref{thm:equiv} we have $\sreg_i(I^n)=\reg(\ginr(I^n) + (x_{i+1},\ldots,x_d))$, so it suffices to prove the claim for the  function $f(n)=\reg(\ginr(I^n) + (x_{i+1},\ldots,x_d))$.
As in \Cref{not:generic}, let $F={\rm Frac}(k[u_{ij}])$, and set $R'=F\otimes_k R=F[x_1,\ldots, x_d], I'=IR'$ and 
$L_i=\sum_{j=1}^d u_{ij}x_j$.  By \Cref{prop:regIvsgin(I)generic} we have the following identities
\begin{equation}
\label{eq:2}
f(n)=\reg(\ginr(I^n) + (x_{i+1},\ldots,x_d))=\reg(I^nR' + (L_{i+1},\ldots,L_d)).
\end{equation}
Now set $\overline{R'}=R'/(L_{i+1},\ldots,L_d)$ and $\overline{I'}=\left(I'+(L_{i+1},\ldots,L_d)\right)/(L_{i+1},\ldots,L_d)$ and note that 
\begin{eqnarray*}
\overline{I'}^{\, n}&=&\left(I'+(L_{i+1},\ldots,L_d)\right)^n/(L_{i+1},\ldots,L_d) \\
&=&\left( I'^{\, n}+(L_{i+1},\ldots,L_d) \right)/(L_{i+1},\ldots,L_d) \\
&=&\left( I^nR'+(L_{i+1},\ldots,L_d) \right)/(L_{i+1},\ldots,L_d) .
\end{eqnarray*}
Thus there is a graded ring isomorphism
\[
\frac{\overline{R'}}{\overline{I'}^{\, n}} \cong \frac{R'}{I^nR'+(L_{i+1},\ldots,L_d)},
\]
which yields equalities 
\[
\reg\left(R'/\overline{I'}^{\, n}\right)=\reg\left(R'/(I^nR'+(L_{i+1},\ldots,L_d))\right) \qquad \text{and} \qquad \reg\left(\overline{I'}^{\, n}\right) =f(n).\] 
Now the eventual linearity of $f(n)$ follows from the results on regularity of powers in \cite{CHT, Kod} applied to $I'$. 
\end{proof}

Naturally, it follows that the other invariants listed \Cref{thm:equiv} also grow linearly.
\begin{cor}
Under the assumptions of \Cref{thm:linear}  the following functions are eventually linear:
\begin{enumerate}
\item $f(n)=a_i(I^n)$ for a fixed integer $0\leq i\leq \codim(I)$,
\item $f(n)=$the largest $i$-th partial degree of a minimal generator of $\ginr(I^n)$. 
\item $f(n)=\red_{d-i}(R/I^n)$ for a fixed integer $0\leq i\leq \codim(I)$.
\end{enumerate}
\end{cor}
This result has been previously shown for the reduction numbers $\red(R/I^n)=\red_{d-\codim(I)}(R/I^n)$ in \cite{Hoa}. To our knowledge the remaining assertions of the above corollary are new.

\subsection*{Acknowledgements} 
This work was completed in the framework of the 2021 Polymath Jr. program (\href{https://geometrynyc.wixsite.com/polymathreu}{geometrynyc.wixsite.com/polymathreu}) supported by NSF award DMS--2113535. The eighth author was supported by NSF grant DMS--2101225.

We thank Robert M.~Walker for useful remarks and for sharing his preprint \cite{Walker} with us.

\bigskip

\bigskip

\end{document}